\documentclass[preprint,3p]{elsarticle}
\journal{and later accepted for
Communications in Combinatorics and Optimization}
\usepackage{graphicx}
\usepackage{amsmath,amssymb,ifthen}

\usepackage{color}
\definecolor{LinkColour}{rgb}{0,0,0.3} % dark blue
\usepackage[colorlinks,citecolor=LinkColour,
linkcolor=LinkColour,urlcolor=LinkColour,
bookmarks=false]{hyperref}

\newtheorem{thm}{Theorem}
\newtheorem{cor}{Corollary}

\newtheorem{lem}{Lemma}

\newcommand{\bihl}[2]{\hypersetup{urlcolor=LinkColour}\href{https://doi.org/#1}{#2}\hypersetup{urlcolor=black}}
\newcommand{\abs}[1]{\left\lvert#1\right\rvert}

\providecommand{\Nset}{\mathbb{N}}
\newcommand{\rb}[1]{_{\rm #1}} % roman subscript 

\newcommand{\floor}[1]{\lfloor#1\rfloor} 
\newenvironment{proof}[1][]%
{\noindent\textit{\ifthenelse{\equal{#1}{}}{Proof}{Proof #1}}: }%
{\leavevmode\unskip\nobreak\quad\hspace*{\fill}$\Box$\par}
\newcommand{\cincgcm}[2]{\centering{\includegraphics[width=#1cm]{#2}}}
\begin{document}
\title{
Combinations without specified separations
}
\author{Michael A. Allen}
\ead{maa5652@gmail.com}
\address{Physics Department, Faculty of Science,
Mahidol University, Rama 6 Road, Bangkok
10400, Thailand}

\begin{abstract}
We consider the restricted subsets of $\mathbb{N}_n=\{1,2,\ldots,n\}$ with
$q\geq1$ being the largest member of the set $\mathcal{Q}$ of
disallowed differences between subset elements. We obtain new results
on various classes of problem involving such combinations lacking specified
separations.  In particular, we find recursion relations for
the number of $k$-subsets for any $\mathcal{Q}$ when
$\abs{\mathbb{N}_q-\mathcal{Q}}\leq2$.  The results are obtained, in a
quick and intuitive manner, as a consequence of a bijection we give
between such subsets and the restricted-overlap tilings of an
$(n+q)$-board (a linear array of $n+q$ square cells of unit width) with
squares ($1\times1$ tiles) and combs.  A
$(w_1,g_1,w_2,g_2,\ldots,g_{t-1},w_t)$-comb is composed of $t$
sub-tiles known as teeth.  The $i$-th tooth in the comb has width
$w_i$ and is separated from the $(i+1)$-th tooth by a gap of width
$g_i$. Here we only consider combs with $w_i,g_i\in\mathbb{Z}^+$.  When
performing a restricted-overlap tiling of a board with such combs and
squares, the leftmost cell of a tile must be placed in an empty cell
whereas the remaining cells in the tile are permitted to overlap other
non-leftmost filled cells of tiles already on the board.
\end{abstract}


\maketitle

\section{Introduction}
The problem of enumerating combinations with disallowed separations is
as follows. We wish to find the number $S_n$ of subsets of
$\mathbb{N}_n=\{1,2,\ldots,n\}$ satisfying the condition
$\abs{x-y}\notin \mathcal{Q}$ where $x,y$ is any pair of elements in
the subset and $\mathcal{Q}$ is a given nonempty subset of
$\mathbb{Z}^+$. We also wish to find the number $S_{n,k}$ of such
subsets of $\Nset_n$ that are of size $k$.  For $\mathcal{Q}=\{1\}$ it
is well known that $S_n=F_{n+2}$ where $F_j$ is the $j$-th Fibonacci
number defined by $F_{j}=F_{j-1}+F_{j-2}+\delta_{j,1}$ with
$F_{j<1}=0$, where $\delta_{i,j}$ is 1 if $i=j$ and 0 otherwise, and
$S_{n,k}=\tbinom{n+1-k}{k}$ \cite{Kap43}. The quantity
$\tbinom{n+1-k}{k}$ is also the number of ways of tiling an
$(n+1)$-board (i.e., an $(n+1)\times1$ board of unit square cells)
using unit squares and $k$ dominoes ($2\times1$ tiles)
\cite{BQ=03}. This correspondence can be regarded as a result of the
bijection given in \cite{AE22}, generalized in
\cite{All22}, and that we will extend to the most general
case here.  It also appears to be well known that if
$\mathcal{Q}=\{1,\ldots,q\}$ then $S_n=S_{n-1}+S_{n-q-1}+\delta_{n+q,0}$ with
$S_{n+q<0}=0$.  Expressions for the number of combinations when
$\mathcal{Q}=\{2\}$, $\mathcal{Q}=\{q\}$ for $q\geq2$, and
$\mathcal{Q}=\{m,2m,\ldots,jm\}$ for $j,m\geq1$ are derived in
\cite{Kon81}, \cite{Pro83,KL91c,KL91b}, and \cite{MS08,All22},
respectively.

$S_{n,k}$ also has links with graph theory. A \textit{path scheme}
$P(n,\mathcal{Q})$ is an undirected graph with vertex set $V=\Nset_n$
and edge set $\{(x,y):\abs{x-y}\in\mathcal{Q}\}$ \cite{Kit06}. A
subset $\mathcal{S}$ of $V$ is said to be an \textit{independent set}
(or a stable set) if no two elements of $\mathcal{S}$ are
adjacent. The number of independent sets of path scheme
$P(n,\mathcal{Q})$ of size $k$ is then clearly $S_{n,k}$ (and the
total number is $S_n$). The elements $q_i$ for
$i=1,\ldots,\abs{\mathcal{Q}}$ of set $\mathcal{Q}$ are said to form a
\textit{well-based sequence} if, when ordered so that $q_j>q_i$ for
all $j>i$, then $q_1=1$ and for all $j>1$ and $\Delta=1,\ldots,q_j-1$,
there is some $q_i$ such that $q_j=q_i+\Delta$
\cite{Val11,Kit06}. Equivalently, the sequence of elements of
$\mathcal{Q}$, the largest of which is $q$, is well based if
$a=\abs{\Nset_q-\mathcal{Q}}$ is zero or if for all $i,j=1,\ldots,a$
(where $i$ and $j$ can be equal), $p_i+p_j\notin\mathcal{Q}$ where the
$p_i$ are the elements of $\Nset_q-\mathcal{Q}$.  E.g., the only
well-based sequences of length 3 are the elements of the sets
$\{1,2,3\}$, $\{1,2,4\}$, $\{1,2,5\}$, and $\{1,3,5\}$. By considering
$P(n,\mathcal{Q})$, Kitaev obtained an expression for the generating
function of $S_n$ when the elements of $\mathcal{Q}$ are a well-based
sequence \cite{Kit06}. We will show the result via combinatorial proof
and also obtain a recursion relation for $S_{n,k}$.

We define a \textit{$(w_1,g_1,w_2,g_2,\ldots,g_{t-1},w_t)$-comb} as a
linear array of $t$ sub-tiles (which we refer to as \textit{teeth}) of
dimensions $w_i\times1$ separated by $t-1$ gaps of width $g_i$.  The
\textit{length} of a comb is $\sum_{i=1}^{t-1}(w_i+g_i)+w_t$.
When the $t$ teeth are all of width $w$ and the $t-1$
gaps are all of width $g$, it is referred to as a $(w,g;t)$-comb \cite{AE23};
such combs can be used to give a combinatorial interpretation of
products of integer powers of two consecutive generalized Fibonacci
numbers \cite{AE24}. Evidently, a $(w,g;1)$-comb (or $w$-comb)
is just a $w$-omino (and a $(w,0;n)$-comb is an $nw$-omino).  A
$(w,g;2)$-comb is also known as a \textit{$(w,g)$-fence}.  The fence
was introduced in \cite{Edw08} to obtain a combinatorial
interpretation of the tribonacci numbers as the number of tilings of
an $n$-board using just two types of tiles, namely, squares and
$(\frac12,1)$-fences. $(\frac12,g)$-fences have also been used to
obtain results on strongly restricted permutations \cite{EA15}.

In this paper, we start in \S\ref{s:bij} by giving the bijection
between combinations with disallowed separations and the
restricted-overlap tilings of boards with squares and combs. Counting
these types of tilings requires knowledge of all permissible minimal
gapless configurations of square-filled and/or restricted overlapping
combs known as metatiles which are introduced in
\S\ref{s:metatiles}. In \S\ref{s:counting} we derive results from
which recursion relations for $S_{n,k}$ (or $S_n$) for various classes
of the set of disallowed differences $\mathcal{Q}$ can be obtained.

\section{The bijection between combinations with disallowed
  separations and restricted-overlap tilings with squares and 
combs}\label{s:bij}

In order to formulate the bijection we first introduce the concept of
the \textit{restricted-overlap tiling} of an $n$-board. In this type of
tiling, any cell but the leftmost cell of a tile is permitted to
overlap any non-leftmost cell of another tile. The $C^2$, $C^2S$, and
$C^3S$ metatiles in Figure~\ref{f:l1r} are examples where such
overlap occurs (note that in that figure and in Figure~\ref{f:1arc},
for clarity,
some of the
  overlapping tiles are shown displaced downwards a little from their
  final position).
It is readily seen that when tiling an $n$-board,
only tiles with gaps can
overlap in this sense and so restricted-overlap tiling of $n$-boards with just
squares and other $w$-ominoes is the same as ordinary tiling.

Let $q$ be the largest element in the set $\mathcal{Q}$ of disallowed
differences.  The comb corresponding to $\mathcal{Q}$ is of length
$q+1$, has cells numbered from 0 to $q$, and is constructed as
follows. By definition, cells 0
and $q$ are filled (as they are the end teeth or parts of
them). For the cells in between, cell $i$ (for $i=1,\ldots,q-1$) is
filled if and only if $i\in\mathcal{Q}$.  For example,
$\mathcal{Q}=\{1,2,4\}$ corresponds to a $(3,1,1)$-comb. One could
also regard it as a $(1,0,2,1,1)$-comb but for simplicity we insist
that all teeth and gaps in a comb corresponding to $\mathcal{Q}$ are
of positive width.  This ensures that there is only one comb that
corresponds to a given $\mathcal{Q}$.

\begin{thm}\label{T:bij}
There is a bijection between the $k$-subsets of $\Nset_n$ each pair
$x,y$ of elements of which satisfy $\abs{x-y}\notin\mathcal{Q}$ and
the restricted overlap-tilings of an $(n+q)$-board using squares and
$k$ combs corresponding to $\mathcal{Q}$ as described above.  
\end{thm}
\begin{proof}
The $(n+q)$-board associated with a subset $\mathcal{S}$ satisfying
the conditions regarding disallowed differences has cells numbered
from 1 to $n+q$. It is obtained by placing a comb at cell $i$
(with the leftmost tooth of the comb occupying cell $i$)
if and
only if $i\in\mathcal{S}$ and then, after all the comb tiles have been
placed, filling any remaining empty cells with squares.
$\mathcal{S}=\varnothing$ therefore corresponds to a board tiled with
squares only. The $n$ singletons correspond to each of the possible
places to put a tile of length $q+1$ on an $(n+q)$-board. If
$\mathcal{S}$ contains more than one element, then there will be a
tiling corresponding to $\mathcal{S}$ iff, for any two elements $x<y$
in $\mathcal{S}$, the comb representing $x$ (which has its cell 0 at
cell $x$ of the board) has a gap at its cell $y-x$ (which is cell $y$
on the board). This will be the case if $y-x\notin\mathcal{Q}$.
\end{proof}

To illustrate the bijection, we return to the $\mathcal{Q}=\{1,2,4\}$
example. The first case of an allowed subset of $\Nset_n$ with more
than one element is $\{1,4\}$ which can only occur if $n\geq4$. This
subset corresponds to the restricted-overlap tiling of an
$(n+4)$-board with the first comb (corresponding to the element 1)
placed at the start (cell number 1) of the board and the second comb
(corresponding to the element 4) placed with its start in the gap of
the first comb which is at cell number 4 of the board. As the
length of a $(3,1,1)$-comb is 5, a board of length at
least 8 is required for such a tiling. The remaining empty cells on
the board (cell 7 and any cells beyond cell 8) are filled with squares.

The bijection also applies in the
$\mathcal{Q}=\varnothing$ case if we set $q=0$. The comb corresponding
to $\mathcal{Q}$ is then a square (but we still call it a comb to
distinguish it from the ordinary squares) so $S_{n,k}$ is the number
of tilings of an $n$-board using $k$ square combs (and $n-k$ ordinary
squares) which is $\tbinom{n}{k}$.

Let $B_n$ be the number of ways to restricted-overlap tile an
$n$-board with squares and combs corresponding to $\mathcal{Q}$, and
let $B_{n,k}$ be the number of such tilings that use $k$ combs. We
choose to set $B_0=B_{0,0}=1$.

\begin{thm}\label{T:S=B}
$S_n=B_{n+q}$ and $S_{n,k}=B_{n+q,k}$.
\end{thm}
\begin{proof}
This follows immediately from Theorem~\ref{T:bij}.
\end{proof}

\begin{lem}\label{L:gf} 
If the number of tilings of an $n$-board with squares and
length-$(q+1)$ combs is given by
\[ 
B_n=\delta_{n,0}+\sum_{m>0}(\alpha_m\delta_{n,m}+\beta_mB_{n-m}),
\quad B_{n<0}=0,
\]
 then the
generating function $G(x)$ for $S_n$ can be written as
\begin{equation}\label{e:Sgf} 
G(x)=\dfrac{1+\sum_{m>0}\left(\alpha_{m+q}+\sum_{j=1}^q\beta_{m+j}\right)x^m}
{1-\sum_{m>0}\beta_mx^m}.
\end{equation} 
\end{lem}
\begin{proof}
The generating function for $B_n$ is
$(1+\sum_{m>0}\alpha_mx^m)/(1-\sum_{m>0}\beta_mx^m)$. The first $q+1$
terms of the expansion of this must be $1+x+\cdots+x^q$ since there
is only one way to tile an $n$-board with squares and combs when
$0\leq n\leq q$, namely, the all-square tiling. From
Theorem~\ref{T:S=B}, $S_n=B_{q+n}$, and so
\[
G(x)=\frac1{x^q}
\left(\dfrac{1+\sum_{m>0}\alpha_mx^m}{1-\sum_{m>0}\beta_mx^m}
-1-x-\cdots-x^{q-1}
\right),
\]
which simplifies to \eqref{e:Sgf} after first putting the terms inside
the parentheses over a common denominator and then in the numerator
discarding the $x^r$ terms for $1\leq r<q$ since they must sum to
zero.
\end{proof}

It can be seen that restricted-overlap tiling with squares and just
one type of $(1,g;t)$-comb, where $g=0,1,2\ldots$, will result in no
overlap of the combs and so the number of such tilings is the same as
for ordinary tiling. For other types of comb, there will be some
tilings where overlap occurs. The case
$\mathcal{Q}=\{m,2m,\ldots,jm\}$ with $j,m\geq 1$ as studied in
\cite{MS08,All22}, which is a generalization of all the
other cases for which results for $S_{n,k}$ were obtained previously,
corresponds to tiling with squares and $(1,m-1;j+1)$-combs. Thus any
results for $S_{n,k}$ we obtain for nonempty $\mathcal{Q}$ not of this form
(and so overlap does occur) will be new ones.

\section{Metatiles}\label{s:metatiles}
We extend the definition of a metatile given in \cite{Edw08,EA15} to
the case of restricted-overlap tiling. A \textit{metatile} is a minimal
arrangement of tiles with restricted overlap that exactly covers a
number of adjacent cells without leaving any gaps.  It is minimal in
the sense that if one or more tiles are removed from a metatile then
the result is no longer a metatile.

When tiling with squares (denoted by $S$) and combs corresponding to
$\mathcal{Q}$ (denoted by $C$), the two simplest metatiles are the
free square (i.e., a square which is not inside a comb), and a comb
with all the gaps filled with squares. The symbolic representation of
the latter metatile is $CS^g$ where $g=\sum_i g_i$. If a comb has no
gaps then it is a $(q+1)$-omino and is therefore a metatile by itself.

\begin{figure} 
\cincgcm{12}{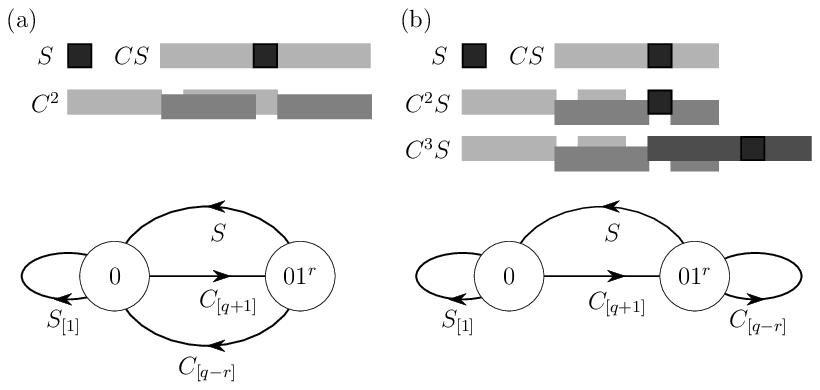}
\caption{Metatiles, their symbolic representations, and digraphs for
  generating them when restricted-overlap tiling an $n$-board with
  squares and $(l,1,r)$-combs for (a)~$r\geq l$ (b)~$r<l$.  
}
\label{f:l1r}
\end{figure}

If the comb contains a gap, we can initiate the creation of at least one more
metatile by placing the start of another comb in the gap. For
instance, in the case of an $(l,1,r)$-comb where $r\geq l$, this is the
only other possible metatile and so there are three metatiles in total
(Figure~\ref{f:l1r}a). However, if $r<l$, adding a comb leaves a
gap which can be filled either with a square to form a completed
metatile ($C^2S$) or with another comb which will leave a further gap,
and so on (Figure~\ref{f:l1r}b). Thus the possible metatiles in
this case are $C^mS$ for $m=0,1,2,\ldots$. More generally, we have the
following lemma.

\begin{lem}\label{L:2r>=q}
Let $r$ be the length of the final tooth in the comb which has at
least one gap. The set of possible metatiles when restricted-overlap
tiling a board of arbitrary length with squares and combs is finite if
and only if $2r\geq q$.
\end{lem}
\begin{proof}
We can reuse the depiction of tilings in 
Figure~\ref{f:l1r} but with the left tooth of each
comb now replaced by arbitrary teeth and gaps (but starting with a
tooth). There is no possibility of a `chain' of combs if, when a second comb
is placed with the first cell of its first tooth just before the start
of the right tooth of the first comb, the start of the right tooth of
the second comb is before or aligned
with the end of the right tooth of the first comb. This occurs if
$q-r\leq r$.  If $q-r>r$, a third comb can be placed so that its first
cell is immediately to the left of the right tooth of the second comb
and this can be continued indefinitely.
\end{proof}

As with fence tiling \cite{EA15}, we can systematically construct all
possible metatiles with the aid of a directed pseudograph (henceforth
referred to as a digraph). As before, the \textit{0 node} corresponds
to the empty board or the completion of the metatile. The other nodes
represent the state of the incomplete metatile.  The occupancy of a
cell in it is represented by a binary digit: 0 for empty, 1 for
filled. We label the node by discarding the leading 1s and trailing
zeros and so the label always starts with a 0 and ends with a 1,
with $1^r$ denoting 1 repeated $r$ times. 
Each
arc represents the addition of a tile and any walk beginning and
ending at the 0 node without visiting it in between corresponds to a
metatile.  With our restricted-overlap tiling, all nodes have an
out-degree of 2 as a gap may always be filled by a square or the start
of a comb. The destination node is obtained by performing a bitwise OR
operation on the bits representing the added tile and the label of
the current node, and then discarding the leading
1s. Figure~\ref{f:l1r} illustrates this for the metatiles
involved when tiling with squares and $(l,1,r)$-combs.

The most important property of a metatile is its length. This is
obtained by summing the contribution to the length associated with
each arc in the walk representing the metatile. The contribution to
the length associated with an arc is zero if it corresponds to the
addition of a square (except in the case of the trivial $S$ metatile)
and for a comb arc equals $q+1-d$ where $d$ is the number of digits in
the node label from which the arc emanates except when that node is
the 0 node in which case $d=0$.
In the digraphs, the contribution to the length associated with each arc is
given as a subscript in square brackets if it is not zero.
For example, from the digraph in
Figure~\ref{f:l1r}b, for tiling with squares and $(l,1,r<l)$-combs we
see that the length of a $C^mS$ metatile where $m>0$ is
$q+1+(m-1)(q-r)=ml+r+1$ as in this case $q=l+r$.

\section{Counting restricted-overlap tilings}
\label{s:counting}

For brevity, we just give results for $B_n$ and $B_{n,k}$ as these are
easily converted to recursion relations for $S_n$ and $S_{n,k}$ and
the generating function for $S_n$ using Theorem~\ref{T:S=B} and
Lemma~\ref{L:gf}.
As with ordinary (non-overlapping) tiling, the following lemma 
is the basis for obtaining 
recursion relations for $B_n$ and $B_{n,k}$.
\begin{lem}\label{L:rr}
For all integers $n$ and $k$,
\begin{subequations}
\label{e:rr} 
\begin{align}\label{e:rrBn}
B_n&=\delta_{n,0}+\sum_{i=1}^{N\rb{m}}B_{n-l_i},\\
\label{e:rrBnk}
B_{n,k}&=\delta_{n,0}\delta_{k,0}+\sum_{i=1}^{N\rb{m}}B_{n-l_i,k-k_i},
\end{align}
\end{subequations}
where $N\rb{m}$ is the number of metatiles, $l_i$ is the length of the
$i$-th metatile and $k_i$ is the number of combs it contains, and
$B_{n<0}=B_{n,k<0}=B_{n<k,k}=0$.
\end{lem}
\begin{proof}
As in \cite{BQ=03,EA15}, we
condition on the last metatile on the board. To obtain \eqref{e:rrBnk}
we note that if an $n$-board tiled with squares and $k$ combs ends
with a metatile of length $l_i$ that contains $k_i$ combs then there
are $B_{n-l_i,k-k_i}$ ways to tile the rest of the board. The
$\delta_{n,0}\delta_{k,0}$ term is from the requirement that
$B_{0,0}=1$. This term arises in the sum when a particular metatile containing
$k$ combs completely fills the board; there is only one tiling where this
occurs. The derivation of \eqref{e:rrBn} is analogous but we ignore
the number of combs. Alternatively, each term in \eqref{e:rrBn} is
obtained from the corresponding term in \eqref{e:rrBnk} by summing
over all $k$.
\end{proof}

\begin{thm}\label{T:2r>=q}
If $\mathcal{Q}=\{1,\ldots,l-1,q-r+1,\ldots,q-1,q\}$ (or
$\mathcal{Q}=\{q-r+1,\ldots,q-1,q\}$ if $l=1$) where $l\geq1$, $2r\geq
q$, and $q+1\geq l+r$, then
\begin{align*} 
B_n&=\delta_{n,0}+B_{n-1}+B_{n-q-1}+\sum_{j=0}^{q-l-r}f^{(l)}_jB_{n-l-q-1-j},\\
B_{n,k}&=\delta_{n,0}\delta_{k,0}+B_{n-1,k}+B_{n-q-1,k-1}
+\sum_{j=0}^{q-l-r}\sum_{i=0}^{\floor{j/l}}\binom{j-(l-1)i}{i}B_{n-l-q-1-j,k-2-i},
\end{align*} 
where the $(1,l)$-bonacci number
$f^{(l)}_j=f^{(l)}_{j-1}+f^{(l)}_{j-l}+\delta_{j,0}$, $f^{(l)}_{j<0}=0$. 
The sums are omitted if $q+1=l+r$.
\end{thm}
\begin{proof} 
We use Lemma~\ref{L:rr}.  If $q+1=l+r$, $C$ is just a $(q+1)$-omino
and the results follow immediately. Otherwise, $C$ is an
$(l,q+1-l-r,r)$-comb. There are two trivial metatiles: $S$ and
$CS^{q+1-l-r}$ which have lengths of 1 and $q+1$, respectively. For
the remaining metatiles, since $2r\geq q$, as described in the proof
of Lemma~\ref{L:2r>=q}, the final comb in the metatile must start
within the gap of the first comb.  Number the cells in this gap from
$j=0$ to $q-l-r$. If the start of the left tooth of the final
comb in the metatile lies in cell $j$ of this gap, the length of the
metatile is $l+q+1+j$. Cells 0 to $j-1$ of the gap can have either an
$S$ or the left tooth of a $C$ and we end up with a metatile of this
length. The number of ways this can be done is simply the number of
ways to tile a $j$-board using squares and $l$-ominoes which is
$f^{(l)}_j$ \cite{BQ=03} (when $l=1$, the $l$-ominoes are regarded as
being distinguishable from the ordinary squares).  Hence there are
$f^{(l)}_j$ metatiles of length $l+q+1+j$ of which
$\tbinom{j-li+i}{i}$ have $2+i$ combs.
\end{proof}

As an example of the application of the above theorem, we consider the
case $\mathcal{Q}=\{2\}$. Then $l=r=1$ and $q=2$ and we obtain the
recursion relation
$B_{n,k}=\delta_{n,0}\delta_{k,0}+B_{n-1,k}+B_{n-2,k-1}+B_{n-4,k-2}$.
Using Theorem~\ref{T:S=B} gives 
$S_{n,k}=\delta_{n,-2}\delta_{k,0}+S_{n-1,k}+S_{n-2,k-1}+S_{n-4,k-2}$
which is in agreement with the recursion relation first obtained by Konvalina
\cite{Kon81}. Note that the metatiles in this case are the
square ($S$), a comb with its gap filled by a square ($CS$),
and two interlocking combs ($C^2$). No overlapping occurs in this case
and so it also an ordinary tiling which has been examined extensively
\cite{EA21}.

For instances where the largest element of $\Nset_q-\mathcal{Q}$ (the
set of allowed differences less than $q$) is $q-r$ and $2r\geq q$ but
the other conditions in Theorem~\ref{T:2r>=q} do not hold, as the
number of possible metatiles is finite (by Lemma~\ref{L:2r>=q}), it is
straightforward to find the length and number of combs in each of them
and then use \eqref{e:rr} to obtain recursion relations for $B_n$ and
$B_{n,k}$.

To enable us to tackle some cases where there are an infinite number
of possible metatiles, we begin by reviewing some terminology
describing features of the digraphs used to construct
metatiles~\cite{EA15}. A \textit{cycle} is a closed walk in which no
node or arc is repeated aside from the starting node.  We refer to
cycles by the arcs they contain. For example, the digraph in
Figure~\ref{f:l1r}(a) has 3 cycles: $S_{[1]}$, $C_{[q+1]}S$, and
$C_{[q+1]}C_{[q-r]}$.  An \textit{inner cycle} is a cycle that does
not include the 0 node. For example, the digraph in
Figure~\ref{f:l1r}(b) has a single inner cycle, namely, $C_{[q-r]}$.  If
a digraph has an inner cycle, there are infinitely many possible
metatiles as, once reached, the cycle can be traversed an arbitrary
number of times before the walk returns to the 0 node.  If all of the
inner cycles of a digraph have one node (or more than one node) in
common, that node (or any one of those nodes) is said to be the
\textit{common node}. In the case of a digraph with one inner cycle,
any of the nodes of the inner cycle can be chosen as the common node.
A \textit{common circuit} is a simple path from the 0 node to the
common node followed by a simple path from the common node back to the
0 node. For example, in the digraph in Figure~\ref{f:l1r}(b), the common
node is $01^r$ and the common circuit is $C_{[q+1]}S$.  If a digraph
has a common node, members of an infinite family of metatiles can be
obtained by traversing the first part of the common circuit from the 0
node to the common node and then traversing the inner cycle(s) an
arbitrary number of times (and in any order if there are more than
one) before returning to the 0 node via the second part of the common
circuit. An \textit{outer cycle} is a cycle that includes the 0 node
but does not include the common node. Thus any metatile which is
not a member of an infinite family of metatiles has a symbolic
representation derived from an outer cycle. E.g., the only outer cycle
in the digraph in Figure~\ref{f:l1r}(b) is $S_{[1]}$.

The following theorem is a restatement of Theorem~5.4 and Identity~5.5
in \cite{EA15} but with more compact expressions for $B_n$ and
$B_{n,k}$ and improved proofs. Note that the length of a cycle or
circuit is simply the total contributions to the length of the arcs it contains.

\begin{thm}\label{T:CN}
For a digraph possessing a common node, let $l_{\mathrm{o}i}$ be the
length of the $i$-th outer cycle ($i=1,\ldots,N\rb{o}$) and
let $k_{\mathrm{o}i}$ be the number of combs it contains, let $L_r$ be
the length of the $r$-th inner cycle ($r=1,\ldots,N$) and let $K_r$ be the
number of combs it contains, and let $l_{\mathrm{c}i}$ be the length
of the $i$-th common circuit ($i=1,\ldots,N\rb{c}$) and
let $k_{\mathrm{c}i}$ be the number of combs it contains. Then for all
integers $n$ and $k$,
\begin{subequations}
\label{e:CN}
\begin{align} 
\label{e:CNBn}
B_n&=\delta_{n,0}+ 
\sum_{r=1}^N (B_{n-L_r}-\delta_{n,L_r})+
\sum_{i=1}^{N\rb{o}}
\biggl(B_{n-l_{\mathrm{o}i}}-\sum_{r=1}^NB_{n-l_{\mathrm{o}i}-L_r}\biggr)
+\sum_{i=1}^{N\rb{c}} B_{n-l_{\mathrm{c}i}},\\
\label{e:CNBnk}
B_{n,k}&=\delta_{n,0}\delta_{k,0}+ 
\sum_{r=1}^N (B_{n-L_r,k-K_r}-\delta_{n,L_r}\delta_{k,K_r})+
\sum_{i=1}^{N\rb{o}}
\biggl(B_{n-l_{\mathrm{o}i},k-k_{\mathrm{o}i}}
-\sum_{r=1}^NB_{n-l_{\mathrm{o}i}-L_r,k-k_{\mathrm{o}i}-K_r}\biggr)\nonumber\\
&\qquad\mbox{}+\sum_{i=1}^{N\rb{c}} B_{n-l_{\mathrm{c}i},k-k_{\mathrm{c}i}},
\end{align} 
\end{subequations}
where $B_{n<0}=B_{n,k<0}=B_{n<k,k}=0$. 
\end{thm}
\begin{proof}
From Lemma~\ref{L:rr},
\begin{subequations}
\begin{align}\label{e:condCNn}
B_n&=\delta_{n,0}+\sum_{i=1}^{N\rb{o}}B_{n-l_{\mathrm{o}i}}+
\sum_{i=1}^{N\rb{c}}\sum_{j_1,\ldots,j_N\ge0}\!\!
\binom{j_1+\cdots+j_N}{j_1,\ldots,j_N}B_{n-\lambda_i},\\
\label{e:condCNnk}
B_{n,k}&=\delta_{n,0}\delta_{k,0}
+\sum_{i=1}^{N\rb{o}}B_{n-l_{\mathrm{o}i},k-k_{\mathrm{o}i}}+
\sum_{i=1}^{N\rb{c}}\sum_{j_1,\ldots,j_N\ge0}\!\!
\binom{j_1+\cdots+j_N}{j_1,\ldots,j_N}B_{n-\lambda_i,k-\kappa_i}
\end{align}
\end{subequations}
with $B_{n<0}=B_{n,k<0}=B_{n<k,k}=0$, 
where $\lambda_i=l_{\mathrm{c}i}+\sum_{s=1}^Nj_sL_s$ and 
$\kappa_i=k_{\mathrm{c}i}+\sum_{s=1}^Nj_sK_s$. 
The multinomial coefficient (which counts the number of
arrangements of the inner cycles) results from the fact that changing
the order in which the inner cycles are traversed (after the common
node is reached via the outgoing path of a common circuit) gives rise to
distinct metatiles of the same length.  The sum of terms over $j_1,\ldots,j_N$
in \eqref{e:condCNn} may be re-expressed as
\[
\sum_{j_1,\ldots,j_N\ge0}\!\!M(\varnothing)B_{n-\lambda_i}=
B_{n-l_{\mathrm{c}i}}+
\sum_{m=0}^{N-1}\sum_{\mathcal{R}_m}
\sum_{\substack{j_{s\notin\mathcal{R}_m}\geq1,\\j_{t\in\mathcal{R}_m}=0}}\!\!\!\!
M(\mathcal{R}_m)B_{n-\lambda_i}
\]
where $M(\mathcal{R})$ denotes the multinomial coefficient
$\tbinom{j_1+\cdots+j_N}{j_1,\ldots,j_N}$ with $j_{t\in\mathcal{R}}=0$,
and $\mathcal{R}_m$ denotes a set of $m$ numbers drawn from $\Nset_N$.  
For example, if $N>2$ an instance of $\mathcal{R}_2$ is $\{1,2\}$ in which case
$M(\mathcal{R}_2)=\tbinom{j_3+\cdots+j_N}{j_3,\ldots,j_N}$. 
Replacing $n$ by $n-L_r$ in
\eqref{e:condCNn} gives
\[
B_{n-L_r}=\delta_{n,L_r}+\sum_{i=1}^{N\rb{o}}B_{n-l_{\mathrm{o}i}-L_r}+
\sum_{i=1}^{N\rb{c}}\sum_{j_1,\ldots,j_N\ge0}\!\!
M(\varnothing)B_{n-\lambda_i-L_r}.
\]
After changing $j_r$ to $j_r-1$, 
the sum of terms over $j_1,\ldots,j_N$ may be re-expressed as
\[
\sum_{\substack{j_r\geq1,\\j_{t\neq r}\geq0}}\!\!\!\!
M_r(\varnothing)B_{n-\lambda_i}
=\sum_{m=0}^{N-1}\sum_{\mathcal{R}_m^{[r]}}
\sum_{\substack{j_{s\notin\mathcal{R}_m^{[r]}}\geq1,\\j_{t\in\mathcal{R}_m^{[r]}}=0}}\!\!\!\!
M_r(\mathcal{R}_m^{[r]})B_{n-\lambda_i},
\]
where $M_r(\mathcal{R})$ denotes the multinomial coefficient
$M(\mathcal{R})$ with $j_r$ replaced by $j_r-1$ (so, for example,
$M_3(\{1,2\})=\tbinom{j_3+\cdots+j_N-1}{j_3-1,\ldots,j_N}$) and
$\mathcal{R}_m^{[r]}$ is a set of $m$ numbers none of which equal $r$ 
drawn from $\Nset_N$.  After subtracting $\sum_{r=1}^N B_{n-L_r}$ from
\eqref{e:condCNn} and using the result for multinomial coefficients
that $M(\mathcal{R})=\sum_{r\notin R}M_r(\mathcal{R})$, we obtain
\eqref{e:CNBn}. Similarly, subtracting $\sum_{r=1}^N B_{n-L_r,k-K_r}$ from
\eqref{e:condCNnk} gives \eqref{e:CNBnk}.
\end{proof}

The rest of the theorems in this section concern families of
$\mathcal{Q}$ whose corresponding digraphs each have a common
node. Once the lengths of and the number of combs in each cycle and
common circuit have been determined, the recursion relations for $B_n$
and $B_{n,k}$ follow immediately from Theorem~\ref{T:CN}.

It can be seen from \eqref{e:CN} that the
recursion relation for $B_n$ can be obtained from that for $B_{n,k}$
by replacing $B_{n-\nu,k-\kappa}$ and
$\delta_{n,\nu}\delta_{k,\kappa}$ by $B_{n-\nu}$ and $\delta_{n,\nu}$,
respectively, for any $\nu$ and $\kappa$. For the remaining theorems
we therefore give the recursion relation for $B_{n,k}$ and only also
show the one for $B_n$ if we use the expression for $B_n$ elsewhere.

For the more generally applicable theorems that follow,
$B_n$ and $B_{n,k}$ are given most simply in terms of the elements
$p_i$ of $\Nset_q-\mathcal{Q}$, the set of allowed differences less
than $q$. We order the $p_i$ so that $p_i<p_{i+1}$ for all
$i=1,\ldots,a-1$, where $a=\abs{\Nset_q-\mathcal{Q}}$, the number of
allowed differences less than $q$. Note that if the comb corresponding
to $\mathcal{Q}$ has leftmost and rightmost teeth of widths $l$ and
$r$, respectively, then $p_1=l$ and $p_a=q-r$.  In digraphs, we let
$\sigma_i$ (for $i=1,\ldots,a$) denote the bit string corresponding to
filling the first $i-1$ empty cells of a comb with squares and
discarding the leading 1s. Thus $\sigma_a$ is always $01^r$.  It is
also easily seen that the comb arc leaving the $\sigma_i$ node is
$C_{[p_i]}$.

\begin{thm}\label{T:wb}
If the elements of $\mathcal{Q}$ are a well-based sequence then
\begin{subequations}
\begin{align}\label{e:wbn}
B_n&=\delta_{n,0}+B_{n-1}+B_{n-q-1}+
\sum_{i=1}^a(B_{n-p_i}-B_{n-p_i-1}-\delta_{n,p_i}),\\
\label{e:wbnk}
B_{n,k}&=\delta_{n,0}\delta_{k,0}+B_{n-1,k}+B_{n-q-1,k-1}+
\sum_{i=1}^a(B_{n-p_i,k-1}-B_{n-p_i-1,k-1}-\delta_{n,p_i}\delta_{k,1}).
\end{align}  
\end{subequations}
If $\mathcal{Q}=\Nset_q$ (and so $a=0$) then the sums over
$i$ are omitted.
\end{thm}
\begin{proof}
If $\mathcal{Q}=\Nset_q$ the comb is a $(q+1)$-omino and the results
follow immediately. Otherwise, we first need to establish that the
comb leaving the $\sigma_i$ node for any $i=1,\ldots,a$ in the digraph
(Figure~\ref{f:wb}; Figure~\ref{f:l1r}(b) shows the $a=1$ instance of the
digraph) takes us back to the $\sigma_1$ node.  This is equivalent to
there being a gap at position $p_i+p_j$ (for any $i,j=1,\ldots,a$) of
the first comb (or that position being beyond the end of the comb)
where $p_j$ can be viewed as the position of the $j$-th empty cell in
the comb added at cell $p_i$ in the first comb. This must be the case
by the definition of a well-based sequence; if there were no gap it
would mean $p_i+p_j\in\mathcal{Q}$ which is impossible.  The digraph
has $a$ inner cycles, namely, $S^{i-1}C_{[p_i]}$ for $i=1,\ldots,a$,
which have lengths $p_i$, respectively. The common node is
$\sigma_1$. There is one common circuit ($C_{[q+1]}S^a$), which is of
length $q+1$, and one outer cycle ($S_{[1]}$), which is of length 1.
\end{proof}

\begin{figure} 
\cincgcm{9}{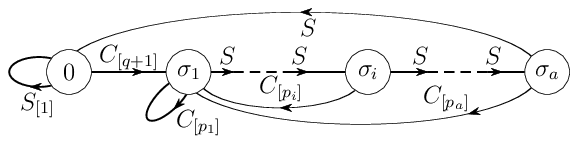}
\caption{Digraph for tiling a board with squares and combs
  corresponding to a well-based sequence of disallowed differences.}
\label{f:wb}
\end{figure}

The following corollary was established via graph theory and results concerning
bit strings in \cite{Kit06}.
\begin{cor}\label{C:wbgf}
The generating function for $S_n$ when the elements $q_i$ of
$\mathcal{Q}$ form a well-based sequence is given by
\begin{equation}\label{e:wbgf} 
G(x)=\frac{c}{(1-x)c-x},
\end{equation}
where $c=1+\sum_{i=1}^{\abs{\mathcal{Q}}}x^{q_i}$.  
\end{cor}
\begin{proof}
Applying Lemma~\ref{L:gf} to \eqref{e:wbn} it can be seen that the
numerator of \eqref{e:Sgf} reduces to
\[
1+\sum_{m=1}^q\biggl(\sum_{j=1}^{q-m}\Bigl(\sum_{i=1}^a(\delta_{m+j,p_i}-\delta_{m+j,p_i+1})\Bigr)+1\biggr)x^m,
\]
where the $+1$ inside the brackets results from the fact that
$\beta_{q+1}$ appears as a term in the sum over $j$ for every $m$ up
to $q$.  Note also that we must have $p_1>1$ and $p_a<q$.  When summed
over $j$, $\delta_{m+j,p_i}-\delta_{m+j,p_i+1}$ cancels (thus leaving
just the $+1$ multiplying the $x^m$) except if $p_i=m$ in which case
$\delta_{m+j,p_i}$ is always zero and the $-\delta_{m+j,p_i+1}$ when
$j=1$ cancels the $+1$.  Hence the numerator simplifies to $c$.  The
denominator of \eqref{e:Sgf} is, in the present case,
$1-x-x^{q+1}-(1-x)\bar{c}$, where $\bar{c}=\sum_{i=1}^ax^{p_i}$. Using
the result that $c+\bar{c}=\sum_{i=0}^qx^i=(1-x^{q+1})/(1-x)$ it is then
easily shown that the denominator can be re-expressed as $(1-x)c-x$.
\end{proof}

We consider the case $\mathcal{Q}=\{1,3,5\}$ as an example for the
application of Theorem~\ref{T:wb}. Then $q=5$, $a=2$, $p_1=2$, and
$p_2=4$. Hence
$B_{n,k}=
\delta_{n,0}\delta_{k,0}-\delta_{n,2}\delta_{k,1}-\delta_{n,4}\delta_{k,1}
+B_{n-1,k}
+B_{n-2,k-1}-B_{n-3,k-1}
+B_{n-4,k-1}-B_{n-5,k-1}
+B_{n-6,k-1}$. Since in this case $c=1+x+x^3+x^5$,
Corollary~\ref{C:wbgf} gives the generating function for $S_n$ of
$(1+x+x^3+x^5)/(1-x-x^2+x^3-x^4+x^5-x^6)$ as found by Kitaev \cite{Kit06}. 

The proofs of some of the results that follow (starting with the next
lemma which is used in the proof of Theorem~\ref{T:min1arc}) require
combs where the number of gaps depends on the particular instance of
$\mathcal{Q}$. We therefore extend our notation: an $(l,[g],r)$-comb
is a comb of length $l+g+r$ whose left and right teeth have widths of
$l$ and $r$, respectively.

\begin{figure}[!b] 
\cincgcm{12}{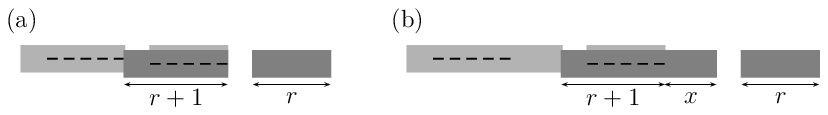}
\caption{The origin of the 1-arc inner cycle at the $01^r$ node when
  restricted-overlap tiling with squares and (a)~$(r+2-g,[g],r)$-combs
where $1\leq g\leq r+1$  
(b)~$(w_1,g_1,\ldots,w_t)$-combs such that $w_t=r$, $g_{t-1}=1$,
  $w_{t-1}\geq q-2r-1=x$, and $t\geq2$.}
\label{f:1arc}
\end{figure}

\begin{lem}\label{L:1arc}
When restricted-overlap tiling an $n$-board with $S$ and $C$, where
$C$ contains at least one gap and the width of the final tooth is
$r$, there is a 1-arc inner cycle $C_{[q-r]}$ containing the $01^r$ node
iff (a)~$q=2r+1$ or (b)~the final gap is of unit width and the
penultimate tooth has a width of at least $q-2r-1$.
\end{lem}
\begin{proof}
If $q<2r+1$, by Lemma~\ref{L:2r>=q}, there can be no inner cycles.  If
$q=2r+1$ and so the length of $C$ is $2(r+1)$, we have the situation
depicted in Figure~\ref{f:1arc}(a). In the figure, the cells of each
comb marked with the dashed line could be parts of teeth or gaps. For
the first (paler) comb depicted in each case, any such cells which are
parts of gaps are filled with other tiles leaving the final gap cell
empty (which corresponds to the $01^r$ node). The start of the next
comb is placed in that cell (and we return to the $01^r$ node).
If $q>2r+1$ the final gap in the
comb must be of unit width and the width $w_{t-1}$ of the penultimate tooth
cannot be less than $x=q+1-2(r+1)$ (Figure~\ref{f:1arc}b).
\end{proof}

\begin{thm}\label{T:min1arc}
Let $\theta$ be the bit string representation of $\mathcal{Q}$ whereby
the $j$-th bit from the right of $\theta$ is 1 if and only if
$j\in\mathcal{Q}$. By $\floor{\theta/2^b}$ we mean discarding the rightmost
$b$ bits in $\theta$ and shifting the remaining bits to the right $b$
places.  Using  $\mid$ to denote the bitwise OR operation, if
$\theta\mid\floor{\theta/2^{p_i-1}}$ for each $i=1,\ldots,a-1$ is all ones
after discarding the leading zeros, $a\ge2$, and $p_a=q-r$ (which
implies that $r\geq1$), then if (a)~$q=2r+1$ or (b)~$q>2r+1$ and $1\leq
p_{a-1}\leq r$, then
\begin{align}\label{e:min1arc} 
B_{n,k}&=\delta_{n,0}\delta_{k,0}-\delta_{n,q-r}\delta_{k,1}
+B_{n-1,k}+B_{n-q+r,k-1}-B_{n-q+r-1,k-1}+B_{n-q-1,k-1}\nonumber\\
&\qquad+\sum_{i=1}^{a-1}(B_{n-q-1-p_i,k-2}-B_{n-2q+r-1-p_i,k-3}).
\end{align} 
\end{thm}

\begin{figure} 
\cincgcm{9}{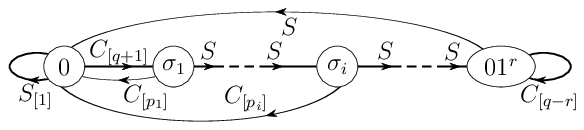}
\caption{Digraph for tiling a board with squares and combs
  corresponding to $\mathcal{Q}$ specified in
  Theorem~\ref{T:min1arc}.}
\label{f:min1arc}
\end{figure}

\begin{proof}
The condition $p_a=q-r$ means that the final tooth is of width $r$.
The conditions (a) and (b) correspond to those in Lemma~\ref{L:1arc}
and thus guarantee a single-comb inner cycle at the $01^r$ node.  The
condition on $\theta$ means that placing a comb at an empty cell
(other than the final empty cell) will result in all gaps in the combs
to the right of this point being filled. On the digraph this means
that there is an arc from the $\sigma_i$ node, where $i=1,\ldots,a-1$,
to the 0 node.  Tiling with squares and combs corresponding to
$\mathcal{Q}$ leads to the digraph shown in
Figure~\ref{f:min1arc}. There is one inner cycle ($C_{[q-r]}$) and one
common circuit ($C_{[q+1]}S^a$). The outer cycles are $S_{[1]}$ and
$C_{[q+1]}S^{i-1}C_{[p_i]}$ for $i=1,\ldots,a-1$ and their respective
lengths are 1 and $q+1+p_i$.
\end{proof}

It is straightforward to verify that the following four classes of
$\mathcal{Q}$ satisfy the conditions for Theorem~\ref{T:min1arc} to
apply: (i)~$\mathcal{Q}=\{2,\ldots,q-r-1,q-r+1,\ldots,q\}$ where
$r\ge1$ and $q\geq\max(2r+1,4)$ (e.g., for $q\leq7$: $\{2,4\}$, $\{2,3,5\}$,
$\{2,4,5\}$, $\{2,3,4,6\}$, $\{2,3,5,6\}$, $\{2,3,4,5,7\}$,
$\{2,3,4,6,7\}$, $\{2,3,5,6,7\}$);
(ii)~$\mathcal{Q}=\{1,\ldots,l-1,l+1,\ldots,q-r-1,q-r+1,\ldots,q\}$
where $r\ge l\geq2$ and $q\geq2r+1$ and $2l\neq q-r$ (e.g., for $q\leq8$:
$\{1,3,4,6,7\}$, $\{1,2,4,6,7,8\}$, $\{1,3,4,5,7,8\}$,
$\{1,3,4,6,7,8\}$);
(iii)~$q=2r+1$, $p_1=l$, $p_a=r+1$, and $l\leq r\leq2l-2$ (e.g., for $q\leq9$:
$\{1,4,5\}$, $\{1,2,5,6,7\}$, $\{1,2,6,7,8,9\}$, $\{1,2,3,6,7,8,9\}$,
$\{1,2,4,6,7,8,9\}$);
(iv)~$\mathcal{Q}=\{2,4,\ldots,2a,2a+1,\ldots,q\}$ where $a\geq3$ and
$q=4a-4,4a-3$ (e.g., for $q\leq9$: $\{2,4,6,7,8\}$, $\{2,4,6,7,8,9\}$).  These
classes cover all cases where the theorem applies for $q\leq9$.

As an example, we consider the case $\mathcal{Q}=\{2,4\}$. Then $q=4$,
$r=1$, $a=2$, $p_1=1$, and $p_2=3$. From \eqref{e:min1arc} we get
$B_{n,k}=
\delta_{n,0}\delta_{k,0}-\delta_{n,3}\delta_{k,1}
+B_{n-1,k}+B_{n-3,k-1}-B_{n-4,k-1}+B_{n-5,k-1}
+B_{n-6,k-2}-B_{n-9,k-3}$.
An explicit formula for $S_{n,k}$ in this case can be obtained in
terms of sums of products of binomial coefficients
\cite{MS08}. Summing over $k$ gives us a recursion relation for $B_n$
whose generating function $(1-x^3)/(1-x-x^3+x^4-x^5-x^6+x^9)$
is that of sequence A224809 in the OEIS
\cite{Slo-OEIS} which does indeed correspond to numbers of subsets
with differences not equalling 2 or 4.

Note that, omitting the sum, Theorem~\ref{T:min1arc} holds for the
case $a=1$ if $p_1=q-r$ and $q\ge2r+1$. It then coincides with
Theorem~\ref{T:wb}.

\begin{thm}\label{T:4l-1}
Suppose $p_1=l$ and $p_a=2l$. Then if either (a)~$q=4l-1$ or
(b)~$p_{a-1}\leq q-2l$ where $q<4l-1$, then
\begin{align} 
B_{n,k}&\!=\!\delta_{n,0}\delta_{k,0}\!-\!\delta_{n,2l}\delta_{k,1}
\!+\!B_{n-1,k}\!+\!B_{n-2l,k-1}\!-\!B_{n-2l-1,k-1}\!+\!B_{n-q-1,k-1}
\!+\!B_{n-q-l-1,k-2}\!+\!B_{n-q-2l-1,k-3}\nonumber\\
&\qquad-B_{n-q-3l-1,k-3}-B_{n-q-4l-1,k-4}
+\sum_{i=2}^{a-1}(B_{n-q-p_i-1,k-2}-B_{n-q-2l-p_i-1,k-3}),
\end{align}
where the sum is omitted if $a=2$. 
\end{thm}
\begin{proof}
Tiling with squares and $(l,[l+1],q-2l)$-combs leads to the digraph
shown in Figure~\ref{f:4l-1}. Note that if $a=2$, the $\sigma_i$ nodes are omitted
  since $i=1,\ldots,a-2$. 
There is just one inner cycle
($C_{[2l]}$) and one common circuit ($C_{[q+1]}S^a$). Their respective
lengths are $2l$ and $q+1$. The outer cycles are $S_{[1]}$,
$C_{[q+1]}C_{[l]}\{S,C_{[l]}\}$, and, if $a\ge3$,
$C_{[q+1]}S^{i-1}C_{[p_i]}$ for $i=2,\ldots,a-1$. Their respective
lengths are 1, $q+l+1$, $q+2l+1$, and $q+p_i+1$.
\end{proof}

The instances of $\mathcal{Q}$ with $q\leq9$ to which
Theorem~\ref{T:4l-1} applies are $\{3\}$, $\{1,3,5,6\}$,
$\{1,5,6,7\}$, $\{1,3,5,6,7\}$, and $\{1,2,4,5,7,8,9\}$.  As an
example, we consider the $\mathcal{Q}=\{3\}$ case. Then $l=1$ and
$a=2$ and Theorem~\ref{T:4l-1} yields
$B_{n,k}=\delta_{n,0}\delta_{k,0}-\delta_{n,2}\delta_{k,1}
+B_{n-1,k}+B_{n-2,k-1}-B_{n-3,k-1}+B_{n-4,k-1}
+B_{n-5,k-2}+B_{n-6,k-3} -B_{n-7,k-3}-B_{n-8,k-4}$.
For the case
$\mathcal{Q}=\{q\}$, Prodinger \cite{Pro83} derived an explicit
formula for $S_{n,k}$ involving the sum of a product of four binomial
coefficients. Konvalina and Liu also showed that
$S_{qm+j}=F_{m+2}^{q-j}F_{m+3}^j$ for $m=0,1,2,\ldots$ and 
$j=0,1,\ldots,q-1$ \cite{KL91c}.  Returning to our $\mathcal{Q}=\{3\}$
result, summing $B_{n,k}$ over all $k$ to obtain a recursion relation
for $B_n$ and then applying Lemma~\ref{L:gf} gives the generating
function for $S_n$ of
$(1+x+x^2+3x^3+x^4-x^5-2x^6-x^7)/(1-x-x^2+x^3-x^4-x^5-x^6+x^7+x^8)$
which matches that for the sequence
$S_{3m+j}=F_{m+2}^{3-j}F_{m+3}^j$ for $m=0,1,2,\ldots$ and $j=0,1,2$ (see
A006500 in the OEIS \cite{Slo-OEIS}).

\begin{figure} 
\cincgcm{9}{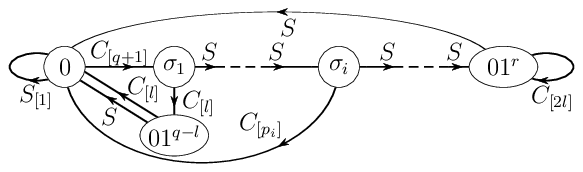}
\caption{Digraph for tiling a board with squares and
  $(l,[l+1],r=q-2l)$-combs used in the proof of
  Theorem~\ref{T:4l-1}.  }
\label{f:4l-1}
\end{figure}

\begin{thm}\label{T:p2}
If $p_1=l$, $p_a=q-r$, $l>r$ and (i)~$q=2l$ or (ii)~$a=2$, $q\geq2l$,
but $q\neq2l+r$, 
then
\begin{align} 
B_{n,k}&=\delta_{n,0}\delta_{k,0}+B_{n-1,k}+B_{n-2l-1,k-1}+B_{n-3l-1,k-2}\nonumber\\
&\qquad+\sum_{i=2}^a(B_{n-p_i,k-1}-B_{n-p_i-1,k-1}+B_{n-l-p_i,k-2}-B_{n-l-p_i-1,k-2}-\delta_{n,p_i}\delta_{k,1}-\delta_{n,l+p_i}\delta_{k,2}).
\end{align} 
\end{thm}
\begin{proof}
Tiling with (i)~squares and $(l,[l-r+1],r)$-combs, where $l>r$, or
(ii)~squares and $(l,1,m\neq l-1,1,r)$-combs, where $0<l-r\leq m+1$,
leads to the digraph shown in Figure~\ref{f:p2}. There are $2(a-1)$
inner cycles: $\{S,C_{[l]}\}S^{i-2}C_{[p_i]}$ for
$i=2,\ldots,a$. Their lengths are $p_i$ and $l+p_i$. The common node
is $\sigma_1$ and so the common circuits are
$C_{[2l+1]}\{S,C_{[l]}\}S^{a-1}$ which have lengths of $2l+1$ and
$3l+1$.
\end{proof}

The instances of $\mathcal{Q}$ for which Theorem~\ref{T:p2} applies
when $q\leq8$ are $\{1,4\}$, $\{1,2,6\}$, $\{1,2,4,6\}$,
$\{1,2,5,6\}$, $\{1,3,4,6\}$, $\{1,2,4,6,7\}$, $\{1,3,4,5,7\}$,
$\{1,2,3,8\}$, $\{1,2,3,5,8\}$, $\{1,2,3,6,8\}$,
$\{1,2,3,5,6,8\}$, $\{1,2,3,7,8\}$, $\{1,2,3,5,7,8\}$, 
$\{1,2,3,6,7,8\}$,
$\{1,2,4,5,6,8\}$, and
$\{1,3,4,5,6,8\}$.

\begin{figure} 
\cincgcm{11}{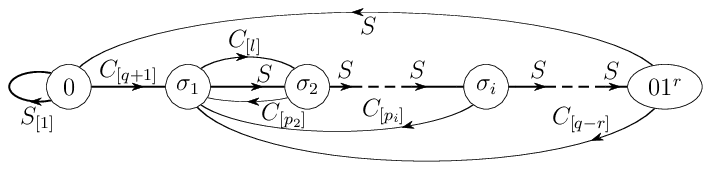}
\caption{Digraph for tiling a board with squares and
  $(l,[l-r+1],r)$-combs, where $l>r$, or with 
squares and $(l,1,m\neq l-1,1,r)$-combs,
  where $0<l-r\leq m+1$, used in the proof of Theorem~\ref{T:p2}.}
\label{f:p2}
\end{figure}

We conclude by showing that all possible $\mathcal{Q}$ such
that $a\leq2$ have been covered by the theorems given here. When
$a=0$, the comb $C$ is a $(q+1)$-omino and $B_{n,k}$ is given by
\eqref{e:wbnk}. When $a=1$, $C$ is an $(l,1,r)$-comb (and the two
possible cases are shown in Figure~\ref{f:l1r}).  Then if $r\geq l$,
Theorem~\ref{T:2r>=q} applies. Otherwise, if $l>r$, Theorem~\ref{T:wb}
applies since $2p_1>q$ (as $p_1=l$ and $q=l+r$) and hence the elements
of $\mathcal{Q}$ are a well-based sequence.  When $a=2$, $C$ is either
an $(l,2,r)$-comb or an $(l,1,m,1,r)$-comb for some $l,m,r\geq1$. In
the former case, $q=l+r+1$ and so the condition $2r\geq q$ leads to
Theorem~\ref{T:2r>=q} applying when $l<r$. When $l=r$, it is covered
by the class~(iii) instances of $\mathcal{Q}$ that apply to
Theorem~\ref{T:min1arc} except when $l=1$ in which case
Theorem~\ref{T:4l-1}(a) applies. When $l=r+1$, we have $q=2l$ and
$p_2=l+1=2l-r$ and so Theorem~\ref{T:p2} applies. The final
possibility is if the elements of $\mathcal{Q}$ are a well-based
sequence (and the case is then covered by Theorem~\ref{T:wb}) and this
occurs if $2p_1>q$ which implies that $l>r+1$. For
$(l,1,m,1,r)$-combs, $q=l+m+r+1$ and so the $2r\geq q$ case
(Theorem~\ref{T:2r>=q}) is when $l<r-m$. Of the $l\geq r-m$ cases we
first consider those where $l=m+1$. This can arise in two ways. If
$2p_1=p_2$ (which implies $l=m+1$) and $p_1+p_2>q$ (which implies
$l>r$) then the elements of $\mathcal{Q}$ are a well-based sequence
and Theorem~\ref{T:wb} applies. When $l=m+1$ and $l\leq r$ then
Theorem~\ref{T:4l-1}(b) applies since these conditions can be
re-expressed as $p_2=2l$ and $p_1\leq r$, respectively (and $l\geq
r-m$ in this case implies $q\leq4l-1$). When $l=1$
(and $l\geq r-m$), the case falls into class~(i) to which
Theorem~\ref{T:min1arc} applies. There are three other ways in which
$l\neq m+1$ arises when $l\geq r-m$. If $l\leq r$ then we have
class~(ii) to which Theorem~\ref{T:min1arc} applies. If $r<l\leq
m+r+1$ then Theorem~\ref{T:p2} applies. Finally, if $2p_1>q$ (which
implies $l>m+r+1$) then the elements of $\mathcal{Q}$ are a well-based
sequence and Theorem~\ref{T:wb} again applies.

\section{Discussion}

As $a=\abs{\Nset_q-\mathcal{Q}}$ increases, so, in general, does the
number of inner cycles in the digraph and we find more and more
instances (e.g., when $\mathcal{Q}=\{1,5\}$ \cite{All-ConnComb}) where
the digraph has inner cycles but no common node. In the simpler of
such cases, it is still possible to derive general recursion relations
analogous to \eqref{e:CN} \cite{All22}. This enables one to find
recursion relations for all the $a=3$ cases, as we will demonstrate in
forthcoming work. For cases where the digraph is more complex, we have
not yet managed to formulate a general procedure for obtaining the
recursion relations.

On looking up sequences $(S_n)_{n\geq0}$ for various choices of
$\mathcal{Q}$ in the OEIS, a number of connections between certain
classes of $\mathcal{Q}$ and some instances of strongly restricted
permutations, combinations, and bit strings were identified. These
are described and proved in \cite{All-ConnComb}.

Various authors have also considered the number of ways of choosing
$k$ objects from $n$ arranged in a circle in such a way that no two
chosen objects are certain disallowed separations apart
\cite{Kap43,Kon81,Mos86,MS08,Guo08,GZ09}. A modified version of our bijection
covers such cases if we instead consider restricted-overlap tiling
using curved squares and combs of a circular $n$-board with
the $n$-th cell joined to the first cell. There are, however,
subtleties about the rules for overlap which we will address in detail
elsewhere.

\subsection*{Acknowledgements}
The author thanks Brian Hopkins for discussions about the manuscript
and the two anonymous referees for their suggestions.

%\bibliography{abbrev,pubabbrev,books,comb}

\end{document}